\documentclass[11pt]{amsart}

\usepackage[T1]{fontenc}
\usepackage[utf8]{inputenc}
\usepackage{lmodern}
\usepackage{microtype}
\usepackage{amsmath,amssymb,mathtools}
\usepackage{enumitem}
\usepackage{booktabs}
\usepackage{xcolor}
\usepackage{array}
\usepackage{hyperref}
\usepackage[nameinlink,capitalize]{cleveref}

\hypersetup{
  colorlinks=true,
  linkcolor=blue!55!black,
  citecolor=blue!55!black,
  urlcolor=blue!55!black
}

\newtheorem{theorem}{Theorem}[section]
\newtheorem{proposition}[theorem]{Proposition}
\newtheorem{lemma}[theorem]{Lemma}

\theoremstyle{definition}

\theoremstyle{remark}

\newtheorem{question}[theorem]{Question}

\crefname{theorem}{Theorem}{Theorems}
\crefname{proposition}{Proposition}{Propositions}
\crefname{lemma}{Lemma}{Lemmas}
\crefname{corollary}{Corollary}{Corollaries}
\crefname{remark}{Remark}{Remarks}
\crefname{question}{Question}{Questions}

\newcommand{\Z}{\mathbb Z}

\newcommand{\Aut}{\operatorname{Aut}}
\newcommand{\Inn}{\operatorname{Inn}}
\newcommand{\Out}{\operatorname{Out}}

\newcommand{\rk}{\operatorname{rk}}

\newcommand{\ab}{\mathrm{ab}}

\title{A Torsion-free Supersoluble Group with Trivial Outer Automorphism Group}

\author{Mattia Brescia --- Ernesto Ingrosso --- Marco Trombetti}


\begin{document}

\begin{abstract}
We give a negative solution to Problem~13.23 of the Kourovka Notebook.
We construct a torsion-free group $G$ of Hirsch length $14$ admitting a
finite series
\[
 1=G_0\triangleleft G_1\triangleleft\cdots\triangleleft G_{14}=G
\]
in which every $G_i$ is normal in $G$ and every factor is infinite
cyclic, but such that $\Out(G)=1$.
\end{abstract}

\maketitle

\section{Introduction}

The existence of non-inner automorphisms is one of the oldest recurring
questions in group theory.  Even within soluble and nilpotent classes the
answer heavily depends on the interaction between the centre, the
abelianization, and the integral structure of the group.  Robinson
constructed infinite soluble groups with no outer automorphisms
\cite{Robinson1980}.  However, he did not settle the much
more restrictive situation in which the group has a normal series with infinite cyclic factors.

This latter question was posed by F.~de Giovanni as Problem~13.23 in the
thirteenth issue of the Kourovka Notebook:
\begin{quote}
Let the group $G$ have a finite normal series with infinite cyclic
factors (containing, of course, $G$ and $\{1\}$).  Is it true that $G$
has a non-trivial outer automorphism?
\end{quote}
The problem, first published in 1995, is still listed without a solution
in the twenty-first issue of the Notebook \cite[Problem~13.23]{Kourovka21}.
The requirement that every member of the series be normal in the whole
group is essential.  It is strictly stronger than merely asking that
$G$ be poly--$\Z$, and it makes the group torsion-free and supersoluble.
For brevity, we call such a group \emph{normally poly--$\Z$}.

The closest previous work is due to Menegazzo and Puglisi
\cite{MenegazzoPuglisi}.  If
\[
 \rho_0(G)=\rk_{\Z}(G/G')
\]
denotes the torsion-free rank of the abelianization, they proved that a
torsion-free supersoluble group has a non-inner automorphism whenever
$\rho_0(G)\geq 2$.  They also constructed torsion-free supersoluble
groups with $\Aut(G)=\Inn(G)$ and $\rho_0(G)=0$.  Their methods leave
precisely the intermediate case $\rho_0(G)=1$ unsolved.  Moreover, if
a torsion-free supersoluble group with infinite abelianization satisfies
$\Aut(G)=\Inn(G)$, then
\[
 Z(G)\cong\Z,\qquad Z(G)\cap G'=1,\qquad \rho_0(G)=1
\]
\cite[Theorem~2]{MenegazzoPuglisi}.  Thus a counterexample to Kourovka
Problem~13.23 must lie exactly on this narrow boundary.

The group constructed here does so.  Our main result is the following.

\medskip

\noindent{\bf Main Theorem}\quad
There exists a torsion-free group with a finite normal series with
infinite cyclic factors, such that $\Out=1$. Moreover $G$ has Hirsch
length $14$, $Z(G)\cong\Z$ and $G_{\ab}\cong \Z\oplus (\Z/2\Z)^6$.

\medskip

The group is built by adding one generator $t$ to $N$ and imposing the rule
\[
G=\langle t\rangle\ltimes_{\alpha}N,
\qquad
t^{-1}nt=\alpha(n).
\]
Here $N$ is torsion-free nilpotent of class $3$ and has Hirsch length $13$.

The automorphism $\alpha$ acts as $-I$ on $N_{\mathrm{ab}}$. However,
$\alpha^2$ is not conjugation by an element of $N$. More concretely,
$\alpha^2$ fixes the $Z_k$'s and $C$, but sends
$X_i$ to $X_iC^{s_i}$, where
$s=(0,1,1,0,1,0)^{\mathsf t}$. Conjugation by a word
$Z_0^{w_0}\cdots Z_5^{w_5}$ changes the $X_i$'s by the vector $Dw$.
The matrix $D$ has determinant $2$, and in this example
$s\notin D\mathbb Z^6$ while $2s\in D\mathbb Z^6$. Thus $\alpha^2$ is not
inner, but $\alpha^4$ is inner. Therefore $[\alpha]$ has order $4$ in
$\operatorname{Out}(N)$.

The remaining point is to show that adding $t$ creates no new outer
automorphisms. Any automorphism of $G$ must preserve $N$. It also cannot send
$tN$ to $t^{-1}N$, because that would force $[\alpha]$ to be conjugate to
$[\alpha]^{-1}$, and this does not happen. Hence $t$ can only be sent to
something of the form $nt$. The calculations in $N$ then show that the
induced automorphism of $N$ must be a power of $\alpha$, up to an inner
automorphism. After removing those inner parts, the only remaining maps fix
$N$ pointwise and send $t$ to $C^jt$. These maps are also inner. Hence every
automorphism of $G$ is inner.

The proof is organized as follows.  In \cref{sec:construction} we define
$N$, while in \cref{sec:dopo} we define  the automorphism $\alpha$, and the group $G$, and verify the required
normal series. In \cref{sec:rigidity} we determine the normalizer of
$\langle[\alpha]\rangle$ in $\Out(N)$.  The
calculation of $\Out(G)$ is completed in \cref{sec:mapping-torus}.  

\section{The group $N$}\label{sec:construction}
Let $\Gamma$ be the graph with vertex set $\{0,1,2,3,4,5\}$ and edge set
\[
\begin{gathered}
 e_0=(0,2),\quad e_1=(1,2),\quad e_2=(1,3),\\
 e_3=(1,4),\quad e_4=(2,4),\quad e_5=(3,5).
\end{gathered}
\] Let $N$ be generated by
\[
X_0,\ldots,X_5,\qquad Z_0,\ldots,Z_5,\qquad C,
\]
subject to the following relations.
\begin{enumerate}[label=\textup{(N\arabic*)},leftmargin=2.7em]
\item\label{rel:N1} $C$ is central, and the elements $Z_0,\ldots,Z_5$
commute pairwise.

\item\label{rel:N2} If $e_k=(i,j)$ is an edge, with $i<j$, then
$[X_i,X_j]=Z_kC^{-1}$. If $\{i,j\}$ is not an edge, then
$[X_i,X_j]=1$.

\item\label{rel:N3} For every vertex $i$ and every edge $e_k$,
\[
[X_i,Z_k]=
 \begin{cases}
 C,&\text{if $i$ is an endpoint of $e_k$},\\
 1,&\text{otherwise.}
 \end{cases}
\]
\end{enumerate}
All commutators of weight at least $4$ are therefore trivial.

In order to prove that  every element of $N$ has a unique normal form
\[
X_0^{a_0}\cdots X_5^{a_5}
Z_0^{b_0}\cdots Z_5^{b_5}C^d,
\qquad a_i,b_k,d\in\mathbb Z
\] and that consequently $N$ is torsion-free nilpotent of class $3$ and has Hirsch length $13$, we need to construct a concrete model for $N$. {This is done by a direct construction from a rational Lie algebra.}

Let $\mathfrak n$ be the rational vector space with basis
\[
x_0,\ldots,x_5,\qquad z_0,\ldots,z_5,\qquad c .
\]
The edge $e_k$ has a corresponding basis element $z_k$. We define the bracket
on basis elements as follows. If $e_k=(i,j)$ is an edge with $i<j$, then
$[x_i,x_j]=z_k$. Also,
\[
[x_i,z_k]=
\begin{cases}
c, & \text{if $i$ is an endpoint of $e_k$},\\
0, & \text{otherwise}.
\end{cases}
\]
All other brackets between basis elements are zero, and the bracket is
extended by bilinearity and alternation. We now check the Jacobi identity on
basis elements.

First, any Jacobi expression containing $c$ is zero, since $c$ is central.

Suppose that at least two entries are edge vectors. Take $u\in\mathfrak n$
and two edge basis elements $z_k,z_l$. We have $[z_k,z_l]=0$. Also,
$[u,z_k]$ and $[u,z_l]$ are either $0$ or a multiple of $c$, hence they are
central. Therefore
\[
[u,[z_k,z_l]]+[z_k,[z_l,u]]+[z_l,[u,z_k]]=0.
\]

Now take two vertex vectors and one edge vector, say $x_i,x_j,z_k$. Then
\[
J(x_i,x_j,z_k)
=
[x_i,[x_j,z_k]]+[x_j,[z_k,x_i]]+[z_k,[x_i,x_j]].
\]
The elements $[x_j,z_k]$ and $[z_k,x_i]$ are either $0$ or multiples of $c$,
so the first two terms are zero. The element $[x_i,x_j]$ is either $0$ or
some $z_l$, and brackets between the $z_l$'s are zero. Hence the third term
is also zero. Thus $J(x_i,x_j,z_k)=0$.

It remains to check triples of vertex vectors. If two entries are equal, the
identity follows from alternation. For example,
\[
\begin{aligned}
J(x_i,x_i,x_j)
&=[x_i,[x_i,x_j]]+[x_i,[x_j,x_i]]+[x_j,[x_i,x_i]] \\
&=[x_i,[x_i,x_j]]-[x_i,[x_i,x_j]]+0=0.
\end{aligned}
\]
So assume that $i,j,l$ are pairwise distinct. We check the first summand
$[x_i,[x_j,x_l]]$. If $\{j,l\}$ is not an edge, then $[x_j,x_l]=0$, so
$[x_i,[x_j,x_l]]=0$. If $\{j,l\}=e_k$ is an edge, then
$[x_j,x_l]=\pm z_k$. The endpoints of $e_k$ are $j$ and $l$. Since $i$ is
different from both, $i$ is not an endpoint of $e_k$, and therefore
$[x_i,z_k]=0$. Hence again $[x_i,[x_j,x_l]]=0$. The same argument gives
$[x_j,[x_l,x_i]]=0$ and $[x_l,[x_i,x_j]]=0$. Therefore
\[
J(x_i,x_j,x_l)
=
[x_i,[x_j,x_l]]+[x_j,[x_l,x_i]]+[x_l,[x_i,x_j]]
=0.
\]
This proves the Jacobi identity.

We now check the nilpotency class. From the definition of the brackets, every
bracket of two basis elements lies in the span of $z_0,\ldots,z_5,c$. Thus
\[
[\mathfrak n,\mathfrak n]\leq \langle z_0,\ldots,z_5,c\rangle.
\]
Bracketing once more with $\mathfrak n$, the only possible nonzero terms are
of the form $[x_i,z_k]$, and these are either $0$ or $c$. Hence
\[
[\mathfrak n,[\mathfrak n,\mathfrak n]]\leq \langle c\rangle.
\]
Since $c$ is central, one more bracket gives
$[\mathfrak n,[\mathfrak n,[\mathfrak n,\mathfrak n]]]=0$.

On the other hand, if $e_k=\{i,j\}$ is an edge, then $[x_i,x_j]=\pm z_k$,
and $i$ is an endpoint of $e_k$. Therefore
\[
[x_i,[x_i,x_j]]
=
\pm[x_i,z_k]
=
\pm c
\neq 0.
\]
So the third term is nonzero, while the fourth term is zero. Hence
\[
\gamma_3(\mathfrak n)=\langle c\rangle\neq0,
\qquad
\gamma_4(\mathfrak n)=0.
\]
Therefore $\mathfrak n$ is nilpotent of class exactly $3$.

We now construct a group $\mathcal N$ from $\mathfrak n$. The elements of
$\mathcal N$ are written as $\exp(u)$, with $u\in\mathfrak n$. The product in
$\mathcal N$ is defined by
\[
\exp(a)\exp(b)
=
\exp\left(
a+b+\frac12[a,b]
+\frac1{12}[a,[a,b]]
-\frac1{12}[b,[a,b]]
\right).
\]
This formula stops here because all brackets of length $4$ are zero. With
this convention, we put
\[
X_i=\exp(x_i),\qquad Z_k=\exp(z_k),\qquad C=\exp(c).
\]
Thus $x_i,z_k,c$ are elements of the Lie algebra, while $X_i,Z_k,C$ are the
corresponding elements of the group $\mathcal N$.

We now compute the group commutator. With our convention,
$[g,h]=g^{-1}h^{-1}gh$. We want to rewrite
$[\exp x,\exp y]$ as $\exp(D)$, with $D\in\mathfrak n$. Put $r=[x,y]$.
First, using the product formula with $a=-x$ and $b=-y$, we get
$A=\log(\exp(-x)\exp(-y))=-x-y+\frac12r-\frac1{12}[x,r]+\frac1{12}[y,r]$.
Next set $B=\log(\exp(A)\exp(x))$. Since brackets of length $4$ vanish, we
have $[A,x]=r-\frac12[x,r]$, $[A,[A,x]]=-[x,r]-[y,r]$, and
$[x,[A,x]]=[x,r]$. Therefore
\[
B
=
A+x+\frac12[A,x]
+\frac1{12}[A,[A,x]]
-\frac1{12}[x,[A,x]]
=
-y+r-\frac12[x,r].
\]
Finally set $D=\log(\exp(B)\exp(y))$. Since $[B,y]=-[y,r]$, and all the
next terms have length $4$, we get $D=B+y+\frac12[B,y]$. Hence
\[
D
=
r-\frac12[x,r]-\frac12[y,r]
=
[x,y]-\frac12[x+y,[x,y]].
\]
Thus
\[
\log[\exp x,\exp y]
=
[x,y]-\frac12[x+y,[x,y]].
\]

We now apply this to the generators. If $e_k=(i,j)$ with $i<j$, then
$[x_i,x_j]=z_k$. Since both $i$ and $j$ are endpoints of $e_k$, we also have
$[x_i,z_k]=c$ and $[x_j,z_k]=c$. Therefore
\[
\log[X_i,X_j]
=
[x_i,x_j]-\frac12[x_i+x_j,[x_i,x_j]]
=
z_k-\frac12([x_i,z_k]+[x_j,z_k])
=
z_k-c.
\]
So $[X_i,X_j]=\exp(z_k-c)=Z_kC^{-1}$.

Similarly, for a vertex $i$ and an edge $e_k$, we have
$\log[X_i,Z_k]=[x_i,z_k]-\frac12[x_i+z_k,[x_i,z_k]]$. If $i$ is not an
endpoint of $e_k$, then $[x_i,z_k]=0$, so $[X_i,Z_k]=1$. If $i$ is an
endpoint of $e_k$, then $[x_i,z_k]=c$, and the correction term is zero
because $c$ is central. Hence $[X_i,Z_k]=C$. Also, $C$ is central. Thus the
generators satisfy the relations \ref{rel:N1}--\ref{rel:N3}.

Let $N$ be the subgroup of $\mathcal N$ generated by
$X_0,\ldots,X_5$, $Z_0,\ldots,Z_5$, and $C$. Using the relations just
computed, every word in these generators can be collected by moving all
$X$'s to the left, all $Z$'s to the middle, and all powers of $C$ to the
right. When an $X_i$ crosses an $X_j$, it only creates powers of some $Z_k$
and of $C$; when an $X_i$ crosses a $Z_k$, it only creates a power of $C$;
and $C$ commutes with everything. Thus every element of $N$ can be written in
the form
\[
X_0^{a_0}\cdots X_5^{a_5}
Z_0^{b_0}\cdots Z_5^{b_5}C^d,
\qquad a_i,b_k,d\in\mathbb Z.
\]
This is the normal form.

We now check that this expression is unique. Suppose a collected word is
trivial:
$X_0^{a_0}\cdots X_5^{a_5}Z_0^{b_0}\cdots Z_5^{b_5}C^d=1$. Take its
logarithm. No bracket can create an $x_i$ term, so the coefficient of $x_i$
in the logarithm is exactly $a_i$. Since the logarithm is zero, all $a_i$ are
zero. After this, the word is $Z_0^{b_0}\cdots Z_5^{b_5}C^d$. Brackets
between the $Z_k$'s are zero, and $C$ is central, so the coefficient of
$z_k$ is exactly $b_k$. Hence all $b_k$ are zero. We are left with $C^d=1$,
whose logarithm is $dc$, so $d=0$. Therefore the normal form is unique.

It follows that the presentation used in the main text is faithful: the
relations give exactly the subgroup $N$ constructed inside $\mathcal N$.
Moreover $N$ is torsion-free. Indeed, if $g^m=1$ for some $m>0$, then
$\log(g^m)=m\log(g)=0$, so $\log(g)=0$, and hence $g=1$.

Consequently $N$ is torsion-free nilpotent of class $3$ and has Hirsch
length~$13$, one generator for each of the six $X_i$'s, the six $Z_k$'s, and
the central generator~$C$.

\section{The group $G$ and its normal series}\label{sec:dopo}

Put
\[
U=\gamma_2(N)=\langle Z_0,\ldots,Z_5,C\rangle,
\qquad
K=\gamma_3(N)=\langle C\rangle.
\]
Then $N/U\cong\mathbb Z^6$ and $U/K\cong\mathbb Z^6$, with bases
$x_i=X_iU$ and $z_k=Z_kK$.

Indeed, every commutator of two generators lies in
$\langle Z_0,\ldots,Z_5,C\rangle$, so
$\gamma_2(N)\leq\langle Z_0,\ldots,Z_5,C\rangle$. Conversely, if
$e_k=(i,j)$, then $Z_k=[X_i,X_j]C\in\gamma_2(N)$, and
$C=[X_i,Z_k]\in\gamma_2(N)$. Hence
$\gamma_2(N)=\langle Z_0,\ldots,Z_5,C\rangle$. Moreover, once we
bracket again with a generator, the only possible non-trivial commutators are
$[X_i,Z_k]$, and these give powers of $C$. Since $C$ is central and
$[X_i,Z_k]=C$ whenever $i$ is an endpoint of $e_k$, we also get
$\gamma_3(N)=\langle C\rangle$.

The centre of $N$ is
\[
Z(N)=K=\langle C\rangle.
\] Let $g\in N$ be central and write its image in $N/U$ as
$X_0^{a_0}\cdots X_5^{a_5}U$. If some $a_j\neq0$, choose an edge
$e_k=\{i,j\}$ incident with $j$. In the commutator with $X_i$, the
$Z_k$-coordinate is $\pm a_j$, and no central element can have such a
non-zero $Z_k$-coordinate. Hence all $a_j$ are zero, and every central
element lies in $U$. It remains to check that no non-zero element of $U/K$
commutes with all the $X_i$.

Take an element of $U/K$ and write it as
$u=Z_0^{b_0}\cdots Z_5^{b_5}K$. To compute $[X_i,u]$, we use the relations
above. The generator $X_i$ commutes with $Z_k$ unless $i$ is an endpoint of
the edge $e_k$. If $i$ is an endpoint of $e_k$, then $[X_i,Z_k]=C$.
Therefore
\[
[X_i,u]=C^{\sum_{k:\, i\in e_k} b_k}.
\]
Here $i\in e_k$ just means that the vertex $i$ is one of the two endpoints
of the edge $e_k$.

The matrix below records exactly these sums. Its rows are indexed by the
vertices $0,\ldots,5$, and its columns are indexed by the edges
$e_0,\ldots,e_5$. The entry in row $i$ and column $k$ is $1$ if $i$ is an
endpoint of $e_k$, and is $0$ otherwise:
\[
D=
 \begin{pmatrix}
 1&0&0&0&0&0\\
 0&1&1&1&0&0\\
 1&1&0&0&1&0\\
 0&0&1&0&0&1\\
 0&0&0&1&1&0\\
 0&0&0&0&0&1
 \end{pmatrix}.
\]
Thus, if $b=(b_0,\ldots,b_5)^{\mathsf t}$, then the $i$-th entry of $Db$ is
$\sum_{k:\, i\in e_k} b_k$. Equivalently, $Db$ records the exponents of $C$
in the six commutators $[X_0,u],\ldots,[X_5,u]$.

This is the only notation we use here: the edge is $e_k$, and the corresponding generator is $Z_k$.

A direct calculation gives $\det D=2$. Hence $Db=0$ with
$b\in\mathbb Z^6$ implies $b=0$. Thus no non-zero element of $U/K$
centralizes all vertex generators.

\medskip

Define a map on the generators of $N$ by
\begin{align*}
 \alpha(X_0)&=X_0^{-1}Z_1,&
 \alpha(X_1)&=X_1^{-1}Z_0Z_3,\\
 \alpha(X_2)&=X_2^{-1}Z_0Z_3,&
 \alpha(X_3)&=X_3^{-1},\\
 \alpha(X_4)&=X_4^{-1}Z_3,&
 \alpha(X_5)&=X_5^{-1},
\end{align*}
while
\begin{equation}\label{eq:alpha-ZC}
 \alpha(Z_4)=Z_4C,
 \qquad
 {\alpha(Z_k)=Z_k\quad(k\neq4)},
 \qquad
 \alpha(C)=C^{-1}.
\end{equation}
Collection with the relations \ref{rel:N1}--\ref{rel:N3} shows that
these images satisfy the defining relations of $N$.  Hence they define
an endomorphism $\alpha$ of $N$.

Set
\[
s=(0,1,1,0,1,0)^{\mathsf t},
\qquad
q=(0,1,0,1,1,0)^{\mathsf t},
\qquad
R=Z_1Z_3Z_4.
\]

We first compute $\alpha^2$. Write $\alpha(X_i)=X_i^{-1}A_i$, where
$A_0=Z_1$, $A_1=Z_0Z_3$, $A_2=Z_0Z_3$, $A_3=1$, $A_4=Z_3$, and
$A_5=1$. Since $\alpha$ fixes $Z_0,Z_1,Z_3$, we have
$\alpha(A_i)=A_i$ for all $i$. Hence
$\alpha^2(X_i)=\alpha(X_i^{-1}A_i)=A_i^{-1}X_iA_i$.

Now $A_i^{-1}X_iA_i=X_i[X_i,A_i]$. Using the relations
$[X_i,Z_k]=C$ exactly when $i$ is an endpoint of $e_k$, we get
\[\begin{array}{c}
[X_0,A_0]=1,\quad
[X_1,A_1]=C,\quad
[X_2,A_2]=C,\quad
[X_3,A_3]=1,\\[0.1cm]
[X_4,A_4]=C,\quad
[X_5,A_5]=1.
\end{array}\]
Therefore
\[
\alpha^2(X_i)=X_iC^{s_i},
\qquad
\alpha^2(Z_k)=Z_k,
\qquad
\alpha^2(C)=C.
\]
For the $Z_k$'s, the only point to check is $Z_4$: indeed
$\alpha^2(Z_4)=\alpha(Z_4C)=Z_4CC^{-1}=Z_4$.

Next we compute $Dq$. Since $q$ has nonzero entries in positions
$1,3,4$, this means that we add the columns of $D$ corresponding to
$e_1,e_3,e_4$. These columns add up to $Dq=2s$.

We now compare $\alpha^4$ with conjugation by $R^{-1}$. Here
$\operatorname{Inn}(R^{-1})$ means the map $g\mapsto R^{-1}gR$.
Since $\alpha^2$ fixes all $Z_k$ and fixes $C$, we have
$\alpha^4(Z_k)=Z_k$ and $\alpha^4(C)=C$. Also,
$\alpha^4(X_i)=\alpha^2(X_iC^{s_i})=X_iC^{2s_i}$.

On the other hand, $R=Z_1Z_3Z_4$, so
$R=Z_0^{q_0}\cdots Z_5^{q_5}$. Hence
$[X_i,R]=C^{(Dq)_i}=C^{2s_i}$. Therefore
$R^{-1}X_iR=X_i[X_i,R]=X_iC^{2s_i}$, while $R^{-1}Z_kR=Z_k$ and
$R^{-1}CR=C$. Thus $\alpha^4$ and conjugation by $R^{-1}$ agree on all
generators, and so
\[
\alpha^4=\operatorname{Inn}(R^{-1}).
\]

In particular $\alpha$ is an automorphism, since its fourth power is an
automorphism. Finally,
$\alpha(R)=\alpha(Z_1Z_3Z_4)=Z_1Z_3(Z_4C)=RC$.

\medskip

Define
\[
G=\langle t\rangle\ltimes_{\alpha} N,
\qquad
t^{-1}nt=\alpha(n)\quad(n\in N).
\]
Thus conjugation by $t^{-1}$ acts on $N$ as $\alpha$. Equivalently, moving an
element of $N$ across $t$ is done by the rule $nt=t\alpha(n)$.

\begin{proposition}\label{prop:normal-series}
The group $G$ is torsion-free and has a normal series of length $14$ whose
factors are all infinite cyclic.
\end{proposition}

\begin{proof}
Consider the chain obtained by adding the generators in this order: first
$C$, then $Z_0,\ldots,Z_5$, then $X_0,\ldots,X_5$, and finally $t$:
\[
\begin{array}{c}
1<\langle C\rangle
<\langle C,Z_0\rangle<\cdots
<\langle C,Z_0,\ldots,Z_5\rangle\\[0.1cm]
<\langle C,Z_0,\ldots,Z_5,X_0\rangle<\cdots<N<G.
\end{array}
\]
Each step adds one new generator, so each factor is generated by the image of
that generator.

We first check normality inside $N$. The generator $C$ is central. The
generators $Z_k$ commute with one another, and conjugating a $Z_k$ by any
$X_i$ changes it only by a power of $C$, since $[X_i,Z_k]$ is either $1$ or
$C$. Similarly, conjugating an $X_i$ by any $X_j$ changes it only by elements
generated by the $Z_k$'s and $C$, since $[X_i,X_j]$ is either $1$ or
$Z_kC^{-1}$. Hence, whenever we conjugate a generator already present in the
chain by a generator of $N$, the result differs from it only by elements
which occur earlier in the chain. Therefore every term of the chain up to
$N$ is normal in $N$.

Now we check the action of $t$. Since $t^{-1}nt=\alpha(n)$, conjugation by
$t^{-1}$ is exactly $\alpha$. The defining formulas for $\alpha$ show that
each subgroup in the chain is sent to itself. Hence the chain is stable under
conjugation by $t^{-1}$, and therefore also under conjugation by $t$. So
every term is normal in $G$.

The normal form shows that the factors up to $N$
are infinite cyclic. The last factor is $G/N\cong\langle tN\rangle\cong
\mathbb Z$, so it is infinite cyclic as well. Thus the series has length
$14$ and all its factors are infinite cyclic.

Finally, suppose that $g\in G$ has finite order. Its image in
$G/N\cong\mathbb Z$ also has finite order. Since $\mathbb Z$ has no nontrivial
torsion, this image is trivial. Hence $g\in N$. But $N$ is torsion-free, so
$g=1$. Therefore $G$ is torsion-free.
\end{proof}

\section{The normalizer of $\langle[\alpha]\rangle$ in $\Out(N)$}\label{sec:rigidity}

The commutator maps are
\[
 \kappa:\Lambda^2(N/U)\longrightarrow U/K,
 \qquad
 \mu:(N/U)\times(U/K)\longrightarrow K,
\]
defined as follows. For $i<j$,
\[
 \kappa(x_i,x_j)=
 \begin{cases}
 z_k,&\text{if $e_k=(i,j)$ is an edge},\\
 0,&\text{otherwise},
 \end{cases}
\]
and $\kappa$ is extended by alternation. Also,
\[
 \mu(x_i,z_k)=
 \begin{cases}
 C,&\text{if $i$ is an endpoint of $e_k$},\\
 1,&\text{otherwise}.
 \end{cases}
\]
After identifying $K=\langle C\rangle$ with $\Z$, the matrix of $\mu$ is $D$.

\begin{lemma}\label{lem:graded-rigidity}
Let $\beta\in\operatorname{Aut}(N)$. On the quotients $N/U$, $U/K$, and $K$,
the maps induced by $\beta$ are respectively
\[
\varepsilon I_6,\qquad I_6,\qquad \varepsilon,
\qquad \varepsilon\in\{1,-1\}.
\]
\end{lemma}

\begin{proof}
Use the bases $x_i=X_iU$ for $N/U$, $z_k=Z_kK$ for $U/K$, and $C$ for $K$.
Let the maps induced by $\beta$ on these three quotients be $S$, $T$, and
multiplication by $\lambda$. Thus $S,T\in \operatorname{GL}_6(\mathbb Z)$
and $\lambda\in\{1,-1\}$.

We write $\kappa$ for the bracket on $N/U$. Thus, if $e_k=(i,j)$ with
$i<j$, then $\kappa(x_i,x_j)=z_k$, and if $\{i,j\}$ is not an edge, then
$\kappa(x_i,x_j)=0$. Since $\beta$ sends brackets to brackets, we have
$T\kappa(u,v)=\kappa(Su,Sv)$.

We also use the brackets between $N/U$ and $U/K$. If
$u=\sum_i a_i x_i$ and $z=\sum_k b_k z_k$, then the exponent of $C$ in
$[u,z]$ is $a^{\mathsf t}Db$. Applying $\beta$ to this bracket multiplies
the exponent of $C$ by $\lambda$. On the other hand, the images of $u$ and
$z$ have coordinate vectors $Sa$ and $Tb$. Therefore, for all $a,b$, we get
$a^{\mathsf t}S^{\mathsf t}DTb=\lambda a^{\mathsf t}Db$. Hence
\[
S^{\mathsf t}DT=\lambda D.
\]

Now write $S=(s_{ri})$, so that
$Sx_i=\sum_r s_{ri}x_r$. If $\{i,j\}$ is not an edge, then
$\kappa(x_i,x_j)=0$, so also $\kappa(Sx_i,Sx_j)=0$. For an edge
$e_k=(p,q)$ with $p<q$, the coefficient of $z_k$ in
$\kappa(Sx_i,Sx_j)$ is
$s_{pi}s_{qj}-s_{qi}s_{pj}$. Therefore every non-edge $\{i,j\}$ gives the
equations
$s_{pi}s_{qj}-s_{qi}s_{pj}=0$ for all edges $e_k=(p,q)$.

Running these equations over the non-edges of our graph gives the following
allowed supports:
\[
\begin{array}{lll}
\operatorname{supp}(Sx_0)\subseteq\{0,1,2,4\},&
\operatorname{supp}(Sx_1)\subseteq\{1\},&
\operatorname{supp}(Sx_2)\subseteq\{2\},\\[0.05cm]
\operatorname{supp}(Sx_3)\subseteq\{3\},&
\operatorname{supp}(Sx_4)\subseteq\{1,2,4\},&
\operatorname{supp}(Sx_5)\subseteq\{1,3,5\}.
\end{array}
\]
These are obtained by comparing the coefficients
$s_{pi}s_{qj}-s_{qi}s_{pj}$ for each non-edge $\{i,j\}$ and each edge
$e_k=(p,q)$. Hence $S$ has the following shape:
\[
S=
\begin{pmatrix}
 d_0&0&0&0&0&0\\
 a_{01}&d_1&0&0&a_{41}&a_{51}\\
 a_{02}&0&d_2&0&a_{42}&0\\
 0&0&0&d_3&0&a_{53}\\
 a_{04}&0&0&0&d_4&0\\
 0&0&0&0&0&d_5
\end{pmatrix}.
\]
This means, for example, that $Sx_0$ may involve $x_0,x_1,x_2,x_4$; that
$Sx_4$ may involve $x_1,x_2,x_4$; and that $Sx_5$ may involve
$x_1,x_3,x_5$. All other entries are forced to be zero by the non-edge
equations above. Since $S$ is invertible, all $d_i$ are nonzero.

We now plug this shape back into the same non-edge equations. The non-edge
$\{0,1\}$ gives
$a_{02}d_1=0$ and $a_{04}d_1=0$. Indeed, these are the coefficients of
$z_1$ and $z_3$ in $\kappa(Sx_0,Sx_1)$. The non-edge $\{0,3\}$ gives
$a_{01}d_3=0$. The non-edge $\{0,4\}$ gives $a_{42}d_0=0$. The non-edge
$\{1,5\}$ gives $a_{53}d_1=0$. The non-edge $\{2,5\}$ gives
$a_{51}d_2=0$. Finally, the non-edge $\{3,4\}$ gives $a_{41}d_3=0$.
Altogether,
\[
a_{02}d_1=a_{04}d_1=a_{01}d_3=a_{42}d_0
=a_{53}d_1=a_{51}d_2=a_{41}d_3=0.
\]
Since all $d_i$ are nonzero, all these $a$'s are zero. Thus $S$ is diagonal:
$S=\operatorname{diag}(d_0,d_1,d_2,d_3,d_4,d_5)$.

Now take an edge $e_k=(i,j)$ with $i<j$. Since $\kappa(x_i,x_j)=z_k$, the
identity $T\kappa(u,v)=\kappa(Su,Sv)$ gives
$Tz_k=\kappa(Sx_i,Sx_j)=\kappa(d_ix_i,d_jx_j)=d_id_jz_k$. Therefore $T$ is
also diagonal: on the basis vector $z_k$ corresponding to the edge
$e_k=(i,j)$, it multiplies by $d_id_j$.

We now use $S^{\mathsf t}DT=\lambda D$. The column of $D$ corresponding to
the edge $e_k=(i,j)$ has two nonzero entries, both equal to $1$: one in row
$i$ and one in row $j$. First $T$ multiplies this whole column by $d_id_j$.
Then $S^{\mathsf t}$ multiplies row $i$ by $d_i$ and row $j$ by $d_j$.
Hence the two nonzero entries become $d_i^2d_j$ and $d_id_j^2$. But in
$\lambda D$ the same two entries are both equal to $\lambda$. Therefore
$d_i^2d_j=\lambda=d_id_j^2$. Since $d_i$ and $d_j$ are nonzero, this gives
$d_i=d_j$.

The graph is connected. Applying the equality $d_i=d_j$ along its edges
gives
$d_0=d_2=d_1=d_3=d_4=d_5$. Let this common value be $d$. Then $S=dI_6$.
Since $S\in\operatorname{GL}_6(\mathbb Z)$, we must have
$d\in\{1,-1\}$. Put $\varepsilon=d$.

For every edge $e_k=(i,j)$, we have $Tz_k=d_id_jz_k=d^2z_k=z_k$. Hence
$T=I_6$. Finally, $S^{\mathsf t}DT=\lambda D$ becomes
$\varepsilon D=\lambda D$, so $\lambda=\varepsilon$. Therefore the maps
induced by $\beta$ on $N/U$, $U/K$, and $K$ are respectively
$\varepsilon I_6$, $I_6$, and $\varepsilon$.
\end{proof}

\begin{proposition}\label{prop:alpha-order}
The class $[\alpha]$ has order exactly $4$ in $\Out(N)$.
\end{proposition}

\begin{proof}
We already computed that $\alpha^4=\operatorname{Inn}(R^{-1})$, with the
convention that $\operatorname{Inn}(h)$ sends $g$ to $hgh^{-1}$. Thus
$\alpha^4$ is inner, so $[\alpha]^4=1$ in $\Out(N)$.

It remains to show that $\alpha^2$ is not inner. Recall that
$\alpha^2(X_i)=X_iC^{s_i}$, while $\alpha^2$ fixes every $Z_k$ and fixes
$C$, where $s=(0,1,1,0,1,0)^{\mathsf t}$.

Suppose that $\alpha^2$ were inner. Since $\alpha^2(X_i)$ differs from
$X_i$ only by a power of $C$, the conjugating element cannot have any
non-zero vertex part: a vertex part would create some $Z_k$-term in the
conjugate of some $X_i$. Therefore, modulo a central element, the conjugating
element must have the form $W=Z_0^{w_0}\cdots Z_5^{w_5}$, with
$w=(w_0,\ldots,w_5)^{\mathsf t}\in\mathbb Z^6$.

Now conjugation by $W^{-1}$ sends $X_i$ to
$W^{-1}X_iW=X_iC^{(Dw)_i}$. Indeed, each $Z_k^{w_k}$ contributes
$C^{w_k}$ exactly when $i$ is an endpoint of the edge $e_k$, and contributes
nothing otherwise. Hence the vector of central exponents produced on the
six generators $X_i$ is exactly $Dw$.

Thus $\alpha^2$ is inner only if $s=Dw$ for some $w\in\mathbb Z^6$. But
every column of $D$ has sum $2$. Therefore the sum of the entries of $Dw$ is
$2(w_0+\cdots+w_5)$, which is even. The sum of the entries of $s$ is
$0+1+1+0+1+0=3$, which is odd. Hence $s\notin D\mathbb Z^6$.

On the other hand, we computed $Dq=2s$, where
$q=(0,1,0,1,1,0)^{\mathsf t}$. This is why $\alpha^4$ is inner: it gives the
central exponent vector $2s$. Therefore $\alpha^2$ is not inner, but
$\alpha^4$ is inner. Hence $[\alpha]$ has order exactly $4$ in $\Out(N)$.
\end{proof}

Let $N^{\mathbb Q}$ be the group obtained from the same rational Lie algebra
as~$N$, but allowing rational coordinates in the normal form. Thus an element
of $N^{\mathbb Q}$ has the form
$X_0^{a_0}\cdots X_5^{a_5}Z_0^{b_0}\cdots Z_5^{b_5}C^d$ with
$a_i,b_k,d\in\mathbb Q$. We put
$\operatorname{Inn}_{\mathbb Q}(N)=\operatorname{Aut}(N)\cap
\operatorname{Inn}(N^{\mathbb Q})$.

\begin{lemma}\label{lem:saturation}
There is an isomorphism
\[
\operatorname{Inn}_{\mathbb Q}(N)/\operatorname{Inn}(N)
\cong D^{-1}\mathbb Z^6/\mathbb Z^6\cong C_2.
\]
The non-trivial element is represented by $[\alpha]^2$.
\end{lemma}

\begin{proof}
Let conjugation by some $g\in N^{\mathbb Q}$ preserve $N$. Write $g$ in
rational normal form. Suppose first that the vertex coordinate of $g$ at
$X_j$ is $v_j$. Choose an edge $e_k$ incident with $j$, say
$e_k=\{i,j\}$. When we conjugate $X_i$ by the vertex part $X_j^{v_j}$, the
normal form acquires a $Z_k$-coordinate equal to $\pm v_j$. Since the image
must lie in $N$, this coordinate must be integral. Hence $v_j\in\mathbb Z$.
This holds for every $j$, because every vertex lies on an edge.

Modulo conjugation by an element of $N$, we may therefore remove all vertex
coordinates of $g$. The central coordinate does not affect conjugation, so we
may write
$g=Z_0^{w_0}\cdots Z_5^{w_5}$ with $w=(w_0,\ldots,w_5)^{\mathsf t}\in
\mathbb Q^6$.

Conjugation by $g^{-1}$ fixes every $Z_k$ and fixes $C$. On the vertex
generators it gives
$g^{-1}X_ig=X_iC^{(Dw)_i}$. Therefore this conjugation preserves $N$ exactly
when all entries of $Dw$ are integers, that is, exactly when
$w\in D^{-1}\mathbb Z^6$.

Two such vectors $w,w'$ give the same element modulo
$\operatorname{Inn}(N)$ exactly when $w-w'\in\mathbb Z^6$, because integral
vectors come from conjugation by words in the integral generators
$Z_0,\ldots,Z_5$. Hence
$\operatorname{Inn}_{\mathbb Q}(N)/\operatorname{Inn}(N)\cong
D^{-1}\mathbb Z^6/\mathbb Z^6$.

A direct row computation gives $|\det D|=2$. Therefore
$D^{-1}\mathbb Z^6/\mathbb Z^6$ has two elements. The vector $q/2$ is not
integral, but $D(q/2)=s$ is integral. Thus $q/2$ represents the non-trivial
class.

In $N^{\mathbb Q}$ we have
$R^{1/2}=Z_1^{1/2}Z_3^{1/2}Z_4^{1/2}$. Since $q/2$ is the exponent vector of
$R^{1/2}$, conjugation by $R^{-1/2}$ sends $X_i$ to $X_iC^{s_i}$ and fixes
all $Z_k$ and $C$. This is exactly the action of $\alpha^2$. Therefore
$[\alpha]^2$ is the non-trivial element of the quotient.
\end{proof}

We need the following elementary calculation.

\begin{lemma}\label{lem:unipotent-vector-space}
Let $\mathcal U$ be the group of automorphisms of $N^{\mathbb Q}$ which act
as the identity on $N^{\mathbb Q}/\gamma_2(N^{\mathbb Q})$, on
$\gamma_2(N^{\mathbb Q})/\gamma_3(N^{\mathbb Q})$, and on
$\gamma_3(N^{\mathbb Q})$. Put
$\mathcal V=\mathcal U/(\mathcal U\cap\operatorname{Inn}(N^{\mathbb Q}))$.
Then $\mathcal V$ is a vector space over $\mathbb Q$. Moreover, conjugation
by an automorphism acting as $-I$ on the first quotient, as $I$ on the second
quotient, and as $-I$ on the third quotient acts on $\mathcal V$ as
multiplication by $-1$.
\end{lemma}

\begin{proof}
Take $\phi\in\mathcal U$. Since $\phi$ is the identity on the three quotients,
it cannot change the image of $X_i$ modulo $\gamma_2$, it cannot change the
image of $Z_k$ modulo $\gamma_3$, and it must fix $C$. Hence we can write
\[
\phi(X_i)=X_iA_iC^{r_i},
\qquad
\phi(Z_k)=Z_kC^{t_k},
\qquad
\phi(C)=C,
\]
where $r_i,t_k\in\mathbb Q$ and
$A_i=Z_0^{a_{0i}}\cdots Z_5^{a_{5i}}$ is a rational word in the $Z_k$'s.

We now see what conditions the numbers $a_{ki}$ and $t_k$ must satisfy. Let
$e_k=(i,j)$ be an edge, with $i<j$. Since $[X_i,X_j]=Z_kC^{-1}$, applying
$\phi$ gives
$\phi([X_i,X_j])=\phi(Z_kC^{-1})=Z_kC^{t_k-1}$. On the other hand,
\[
[\phi(X_i),\phi(X_j)]
=
[X_iA_iC^{r_i},X_jA_jC^{r_j}].
\]
The powers of $C$ do not matter in this commutator, because $C$ is central.
Also, the $A_i$'s are made only from the $Z_k$'s, and the $Z_k$'s commute
with each other. Thus the only extra central terms come from moving $X_i$
past $A_j$ and moving $A_i$ past $X_j$. If $a_i=(a_{0i},\ldots,a_{5i})^{\mathsf t}$,
then these two contributions are $(Da_j)_i$ and $-(Da_i)_j$. Therefore the
edge relation gives
\[
t_k=(Da_j)_i-(Da_i)_j.
\]

If $\{i,j\}$ is not an edge, then $[X_i,X_j]=1$. Applying the same computation
gives
\[
0=(Da_j)_i-(Da_i)_j.
\]
Thus all the conditions on the numbers $a_{ki}$ and $t_k$ are linear
equations over $\mathbb Q$.

Conversely, if rational numbers $a_{ki}$ and $t_k$ satisfy these equations,
then the formulas above preserve all relations between the generators. Hence
they define an automorphism of $N^{\mathbb Q}$. The inverse is obtained by
using the opposite parameters for the $Z$-parts and for the changes
$Z_k\mapsto Z_kC^{t_k}$, possibly with an additional central change on the
$X_i$'s. This additional central change is inner in $N^{\mathbb Q}$, as
explained below. Therefore the possible pairs $(a_{ki},t_k)$
form a vector space over $\mathbb Q$.

Now look at composition. If $\phi$ has parameters $(a_{ki},t_k)$ and $\psi$
has parameters $(a'_{ki},t'_k)$, then $\phi\psi$ has parameters
$(a_{ki}+a'_{ki},t_k+t'_k)$. Indeed, the $Z$-part added to $X_i$ is just the
product of the two $Z$-parts, and the central change on $Z_k$ is
$C^{t_k}C^{t'_k}=C^{t_k+t'_k}$. Thus the group law on these parameters is
addition.

The remaining numbers $r_i$ describe maps of the form $X_i\mapsto X_iC^{r_i}$.
These do not affect the equations above. They are also inner in
$N^{\mathbb Q}$. Indeed, since $D$ is invertible over $\mathbb Q$, for any
$r=(r_0,\ldots,r_5)^{\mathsf t}$ we can choose $w\in\mathbb Q^6$ with
$Dw=r$. Then conjugation by $Z_0^{-w_0}\cdots Z_5^{-w_5}$ sends
$X_i$ to $X_iC^{r_i}$ and fixes all $Z_k$ and $C$. So the $r_i$'s disappear
after quotienting by inner automorphisms.

Inner automorphisms themselves also give a linear subspace of the same
parameter space: conjugating by a rational normal word gives parameters
which depend linearly on the exponents of that word. Therefore, after
quotienting by inner automorphisms, we still have a vector space over
$\mathbb Q$. This proves that $\mathcal V$ is a $\mathbb Q$-vector space.

Now let $\sigma$ be an automorphism which acts as $-I$ on the first quotient,
as $I$ on the second quotient, and as $-I$ on the third quotient. We compute
what conjugation by $\sigma$ does to the parameters of $\phi$.

First consider a $Z$-part in the image of some $X_i$. For example, suppose a
term $Z_k^a$ occurs in $A_i$. Since $\sigma^{-1}$ changes $X_i$ to
$X_i^{-1}$ modulo lower terms, applying $\phi$ changes the sign of this
$Z_k$-exponent. Then $\sigma$ fixes the $Z_k$-direction, because it acts as
$I$ on the second quotient. Hence the exponent $a$ becomes $-a$.

Next consider the central change on $Z_k$. If $\phi(Z_k)=Z_kC^b$, then
$\sigma^{-1}$ fixes the $Z_k$-direction, while $\sigma$ sends $C$ to
$C^{-1}$. Hence the exponent $b$ becomes $-b$.

Finally, a direct central change $X_i\mapsto X_iC^{r_i}$ keeps the same sign,
because both the $X_i$-direction and the $C$-direction get a minus sign. But
this part is inner in $N^{\mathbb Q}$, as shown above, so it is zero in
$\mathcal V$.

Thus conjugation by $\sigma$ sends every remaining parameter to its negative.
Therefore it acts on $\mathcal V$ as multiplication by $-1$.
\end{proof}

\begin{proposition}\label{prop:normalizer}
The subgroup generated by $[\alpha]$ is self-normalizing in $\Out(N)$:
\[
\operatorname{N}_{\Out(N)}(\langle[\alpha]\rangle)
=
\langle[\alpha]\rangle
\cong C_4.
\]
Moreover $[\alpha]$ is not conjugate to $[\alpha]^{-1}$ in $\Out(N)$.
\end{proposition}

\begin{proof}
Let $[\beta]$ normalize $\langle[\alpha]\rangle$. By
\cref{lem:graded-rigidity}, the automorphism $\beta$ acts on the three
quotients either as $(I,I,I)$ or as $(-I,I,-I)$. Since $\alpha$ acts as
$(-I,I,-I)$, replacing $\beta$ by $\beta\alpha^{-1}$ if necessary lets us
assume that $\beta$ acts as the identity on the three quotients.

Now work in $N^{\mathbb Q}$. There, $\alpha^2$ is conjugation by
$R^{-1/2}$, so the image of $[\alpha]$ in $\Out(N^{\mathbb Q})$ has order
$2$. Since $[\beta]$ normalizes the subgroup generated by $[\alpha]$, its
image must commute with the image of $[\alpha]$ in $\Out(N^{\mathbb Q})$.

Let $v$ be the class of $\beta$ in $\mathcal V$. By the previous lemma,
conjugation by $\alpha$ sends $v$ to $-v$. But $\beta$ commutes with
$\alpha$ in $\Out(N^{\mathbb Q})$, so the same class must also be $v$. Hence
$v=-v$. Since $\mathcal V$ is a vector space over $\mathbb Q$, this gives
$v=0$. Therefore $\beta$ is inner in $N^{\mathbb Q}$.

By \cref{lem:saturation}, an automorphism of $N$ which is inner in
$N^{\mathbb Q}$ represents, modulo $\operatorname{Inn}(N)$, either the
trivial class or the class $[\alpha]^2$. Thus, after undoing the possible
multiplication by $\alpha^{-1}$ at the start, we get
$[\beta]\in\langle[\alpha]\rangle$.

The reverse inclusion is automatic: every power of $[\alpha]$ normalizes the
subgroup generated by $[\alpha]$. Hence the normalizer is exactly
$\langle[\alpha]\rangle$.

Finally, suppose that some element of $\Out(N)$ conjugated $[\alpha]$ to
$[\alpha]^{-1}$. Then that element would normalize
$\langle[\alpha]\rangle$, so it would have to lie in
$\langle[\alpha]\rangle$. But this subgroup is cyclic, hence its elements
commute with $[\alpha]$ and cannot send it to $[\alpha]^{-1}$. Since
$[\alpha]$ has order $4$, we have $[\alpha]\neq[\alpha]^{-1}$. Therefore
$[\alpha]$ is not conjugate to $[\alpha]^{-1}$.
\end{proof}

\section{The proof of the main theorem}\label{sec:mapping-torus}

Recall that
\[
G=\langle t\rangle\ltimes_{\alpha} N,
\qquad
t^{-1}nt=\alpha(n)\quad(n\in N).
\]
On $N/U$, the automorphism $\alpha$ acts as $-I$. Hence, in the
abelianization of~$G$, we get $x_i=-x_i$ for each $i$. Equivalently,
$2x_i=0$. The elements coming from $U$ already vanish in $N_{\mathrm{ab}}$,
because $U=[N,N]$. Therefore
\[
G_{\mathrm{ab}}\cong \mathbb Z\oplus(\mathbb Z/2\mathbb Z)^6.
\]
The infinite cyclic factor is generated by the image of $t$.

Define
\[
\ell:G\longrightarrow \mathbb Z,
\qquad
\ell(t)=1,
\qquad
\ell(N)=0.
\]
Any homomorphism from $G$ onto $\mathbb Z$ must kill the finite
$(\mathbb Z/2\mathbb Z)^6$ part of $G_{\mathrm{ab}}$, and must send the
generator of the free $\mathbb Z$-part to either $1$ or $-1$. Thus the only
surjective homomorphisms $G\to\mathbb Z$ are $\ell$ and $-\ell$. Hence
\[
N=\ker \ell
\]
is preserved by every automorphism of $G$.

\medskip

Set
\[
S=R^{-1}t^4.
\]
We claim first that $S$ commutes with every element of $N$. Indeed,
$t^4$ acts on $N$ as $\alpha^{-4}$. Since $\alpha^4=\operatorname{Inn}(R^{-1})$,
we have $\alpha^{-4}=\operatorname{Inn}(R)$. Thus, for $n\in N$,
$t^4nt^{-4}=RnR^{-1}$. Therefore
$SnS^{-1}=R^{-1}(t^4nt^{-4})R=R^{-1}(RnR^{-1})R=n$.

We shall also use the following two formulas. Since $t^{-1}Rt=\alpha(R)=RC$
and $t^{-1}Ct=C^{-1}$, we get
\[
t^{-1}St=C^{-1}S,
\qquad
t^{-1}Ct=C^{-1}.
\]

\begin{lemma}\label{lem:centralizer}
We have $C_G(N)=\langle C,S\rangle$.
Every automorphism of $G$ which fixes $N$ pointwise is inner.
\end{lemma}

\begin{proof}
Take an element $nt^k\in G$, with $n\in N$. Its conjugation action on $N$ is
the action of $n$, followed by the action of $t^k$. Since $t^{-1}$ acts as
$\alpha$, the element $t^k$ acts as $\alpha^{-k}$. Therefore the action of
$nt^k$ on $N$ is, modulo inner automorphisms, the class $[\alpha]^{-k}$.

If $nt^k$ commutes with every element of $N$, this action is trivial. Hence
$[\alpha]^{-k}=1$ in $\operatorname{Out}(N)$. Since $[\alpha]$ has order $4$,
we get $4\mid k$. Write $k=4r$.

Now $t^{4r}$ acts on $N$ as $\alpha^{-4r}=\operatorname{Inn}(R^r)$. Hence
$nt^{4r}$ acts on $N$ as conjugation by $nR^r$. This action is trivial
exactly when $nR^r\in Z(N)=\langle C\rangle$. So $n=C^jR^{-r}$ for some
$j\in\mathbb Z$. Since $t^4$ commutes with $R$, we get
$nt^{4r}=C^jR^{-r}t^{4r}=C^jS^r$. Thus every element commuting with $N$
lies in $\langle C,S\rangle$. The reverse inclusion is clear from the
definition of $S$ and from $C\in Z(N)$.

Now let $\varphi$ be an automorphism of $G$ which fixes $N$ pointwise. Its
image on $G/N\cong\mathbb Z$ cannot send $tN$ to $-tN$, because that would
replace the action $\alpha$ by $\alpha^{-1}$, and $\alpha^2$ is not inner.
Hence $\varphi(t)=nt$ for some $n\in N$.

Apply $\varphi$ to the relation $t^{-1}ut=\alpha(u)$. Since $\varphi$ fixes
$u$ and $\alpha(u)$, we get $(nt)^{-1}u(nt)=\alpha(u)$ for every $u\in N$.
This says that conjugation by $n$ does not change the action on $N$, so
$n\in Z(N)=\langle C\rangle$. Thus $\varphi(t)=C^jt$.

Finally, from $t^{-1}St=C^{-1}S$ we get $StS^{-1}=Ct$. Therefore the map
which fixes $N$ and sends $t$ to $Ct$ is conjugation by $S$. Its $j$-th power
sends $t$ to $C^jt$. Hence every automorphism fixing $N$ pointwise is inner.
\end{proof}

\begin{proposition}\label{prop:center-G}
The centre of $G$ is $Z(G)=\langle CS^{-2}\rangle\cong\mathbb Z$, and \hbox{$Z(G)\cap G'=1$.}
\end{proposition}

\begin{proof}
A central element must commute with $N$, so by the previous lemma it has the
form $C^jS^r$. We now ask when it also commutes with $t$. Using the formulas
above,
$t^{-1}C^jS^rt=C^{-j-r}S^r$. This equals $C^jS^r$ exactly when
$-j-r=j$, that is, when $r=-2j$. Hence every central element is a power of
$CS^{-2}$.

The element $CS^{-2}$ has nonzero image in the free $\mathbb Z$-summand of
$G_{\mathrm{ab}}$, because $S^{-2}$ contains $t^{-8}$. Therefore no nonzero
power of $CS^{-2}$ lies in $G'$. Hence $Z(G)\cap G'=1$.
\end{proof}

\medskip

\noindent{\it Proof of the Main Theorem}\quad
The group $G$ is torsion-free and has the required normal series by
\cref{prop:normal-series}. Its abelianization was computed above, and its
centre is given by \cref{prop:center-G}. It remains to show that every
automorphism of $G$ is inner.

Let $\Phi\in\operatorname{Aut}(G)$. Since $N$ is preserved by every
automorphism, $\Phi$ restricts to an automorphism $\beta$ of $N$. Also,
on $G/N\cong\mathbb Z$, the automorphism $\Phi$ sends $tN$ either to $tN$ or
to $t^{-1}N$. Hence
$\Phi(t)=nt^\varepsilon$ for some $n\in N$ and some
$\varepsilon\in\{1,-1\}$.

Apply $\Phi$ to the relation $t^{-1}ut=\alpha(u)$. Ignoring inner
automorphisms of $N$, this gives
\[
[\beta][\alpha][\beta]^{-1}=[\alpha]^\varepsilon.
\]
If $\varepsilon=-1$, then $[\beta]$ conjugates $[\alpha]$ to
$[\alpha]^{-1}$, which is impossible by \cref{prop:normalizer}. Therefore
$\varepsilon=1$.

Again by \cref{prop:normalizer}, we have $[\beta]\in\langle[\alpha]\rangle$.
So $[\beta]=[\alpha]^k$ for some $k\in\mathbb Z$. Conjugating $\Phi$ by a
suitable power of $t$ changes its restriction to $N$ by a power of $\alpha$.
After doing this, and then conjugating by a suitable element of $N$, we may
assume that the new automorphism fixes $N$ pointwise.

By \cref{lem:centralizer}, any automorphism which fixes $N$ pointwise is
inner. The two changes we made were also inner automorphisms. Therefore the
original automorphism $\Phi$ was inner. Hence
\[
\operatorname{Aut}(G)=\operatorname{Inn}(G),
\qquad
\operatorname{Out}(G)=1.
\] The statement is proved.

\medskip

In view of our main result, the following question appears natural.

\begin{question}
What is the minimum Hirsch length of a normally poly--$\Z$ group with
trivial outer automorphism group?
\end{question}

\begin{flushleft}
\rule{8cm}{0.4pt}\\
\end{flushleft}

{
\sloppy
\noindent
Mattia Brescia, Ernesto Ingrosso, Marco Trombetti

\noindent 
Dipartimento di Matematica e Applicazioni ``Renato Caccioppoli''

\noindent
Università di Napoli Federico II

\noindent
Complesso Universitario Monte S. Angelo

\noindent
Via Cintia, Napoli (Italy)

\noindent
e-mail: mattia.brescia@unina.it; ernesto.ingrosso2@unina.it\\ \phantom{e-mail:} marco.trombetti@unina.it 

}

\end{document}